\documentclass[12pt,reqno]{amsproc}

\usepackage[margin=0.9in]{geometry}
\usepackage{nicefrac, graphicx}
\usepackage{amsmath, amssymb, amsthm}

\usepackage[sc]{mathpazo}  \renewcommand{\mathbf}{\mathbold}

\usepackage{mathrsfs}
\newcommand{\arxiv}[1]{}

\usepackage{tikz}
\usetikzlibrary{matrix,arrows}

\usepackage[all,2cell,cmtip]{xy}
\UseAllTwocells

\usepackage{hyperref}

\hypersetup
{
    colorlinks,	%
    citecolor=blue,%
    filecolor=black,%
    linkcolor=red,%
    urlcolor=blue
}

\theoremstyle{plain}
  \newtheorem{theorem}{Theorem}[section]
  \newtheorem{corollary}[theorem]{Corollary}
  \newtheorem{lemma}[theorem]{Lemma}
  
  \newtheorem{proposition}[theorem]{Proposition}

\theoremstyle{definition}
  \newtheorem{definition}[theorem]{Definition}
  
  \newtheorem{ex}[theorem]{Example}
  
  \newtheorem{remark}[theorem]{Remark}
  \newenvironment{example}{\begin{ex}}{\end{ex}}

  \newcommand{\set}[1]{\left\{#1\right\}}

  \newcommand{\R}{\mathbb{R}}

  \newcommand{\Z}{\mathbb{Z}}

  \newcommand{\category}[1]{\mathrm{\mbox{\bf #1}}}
	\newcommand{\Face}{{\category{Fc}}}

	\newcommand{\Ch}{{\category{Ch}}}
	
	\newcommand{\Mod}{\category{Mod}}

		\newcommand{\codim}{\text{\em codim}}

	\newcommand{\op}{} 
	\newcommand{\HBM}{{H}_\mathrm{c}}
	\newcommand{\CBM}{{C}_\mathrm{c}}

	\newcommand{\bA}{\mathbf{A}}
	
	\newcommand{\bC}{\mathbf{C}}

	\newcommand{\bF}{\mathbf{F}}
	\newcommand{\bG}{\mathbf{G}}

	\newcommand{\bK}{\mathbf{K}}
	\newcommand{\bL}{\mathbf{L}}

		\newcommand{\bP}{\mathbf{P}}

	\newcommand{\bS}{\mathbf{S}}

	\newcommand{\cL}{\mathrm{L}}

	\newcommand{\cU}{\mathrm{U}}
	\newcommand{\cV}{\mathrm{V}}
	\newcommand{\cX}{\mathrm{X}}
	\newcommand{\cY}{\mathrm{Y}}
	\newcommand{\cZ}{\mathrm{Z}}

	\newcommand{\surj}{\twoheadrightarrow}
	\newcommand{\inj}{\hookrightarrow}
	
	\newcommand{\Cone}{\mathscr{C}}

\newcommand{\st}{\category{st}}
\newcommand{\lk}{\category{lk}}

	\newcommand{\ip}[1]{\langle{#1}\rangle}

	\newcommand{\Fc}[1]{\Fc[#1]} 

	\newcommand{\inc}{\hookrightarrow}
	
	\newcommand{\pro}{\twoheadrightarrow}

 \newcommand{\under}{\mathbin{\mkern-3mu/\mkern-6mu/}\mkern-3mu}
	\newcommand{\fiber}{\under}

\newcommand{\tab}{\hspace{15pt}}

\newcommand{\geqdn}{\mathrel{\text{\rotatebox[origin=c]{-90}{$\geq$}}}}
\newcommand{\roteq}{\mathrel{\text{\rotatebox[origin=c]{90}{$=$}}}}

\title{{Local Cohomology and Stratification}}
\author{Vidit Nanda}
\date{}

\begin{document}

\begin{abstract}
We outline an algorithm to recover the canonical (or, coarsest) stratification of a given finite-dimensional regular CW complex into cohomology manifolds, each of which is a union of cells. The construction proceeds by iteratively localizing the poset of cells about a family of subposets; these subposets are in turn determined by a collection of cosheaves which capture variations in cohomology of cellular neighborhoods across the underlying complex. The result is a nested sequence of categories, each containing all the cells as its set of objects, with the property that two cells are isomorphic in the last category if and only if they lie in the same canonical stratum. The entire process  is amenable to efficient distributed computation.

\medskip
{\footnotesize
\noindent {\em Communicated by Herbert Edelsbrunnner.}
}
\end{abstract}

\maketitle

\vspace{-.5in}

\section{Introduction}

Setting aside technicalities for the moment, an $n$-dimensional {\bf stratification} of a given topological space $\cX$ is a nested sequence of closed subspaces
\[
\varnothing = \cX_{-1} \subset \cX_0 \subset \cdots \subset \cX_n = \cX
\]
where each successive difference $\cX_d - \cX_{d-1}$ resembles a (possibly empty) $d$-dimensional manifold whose connected components are the $d$-dimensional {\em strata}. The details which have been  suppressed are designed to guarantee the uniformity of small neighborhoods around points in a given stratum. Stratified spaces are of fundamental importance in any branch of mathematics which seriously confronts singularities. Every topological manifold admits a straightforward stratification into its connected components; less trivially, the following spaces all admit, and are often analyzed using, stratifications: finite CW complexes, quotients of properly discontinuous Lie group actions on smooth manifolds, (semi)algebraic varieties and (sub)analytic sets.

The $2$-dimensional singular space below -- let's call it $\cY$ -- is built by pinching a torus along a meridian and attaching a disk across an equator:
\begin{figure}[h!]
\includegraphics[scale = 0.3]{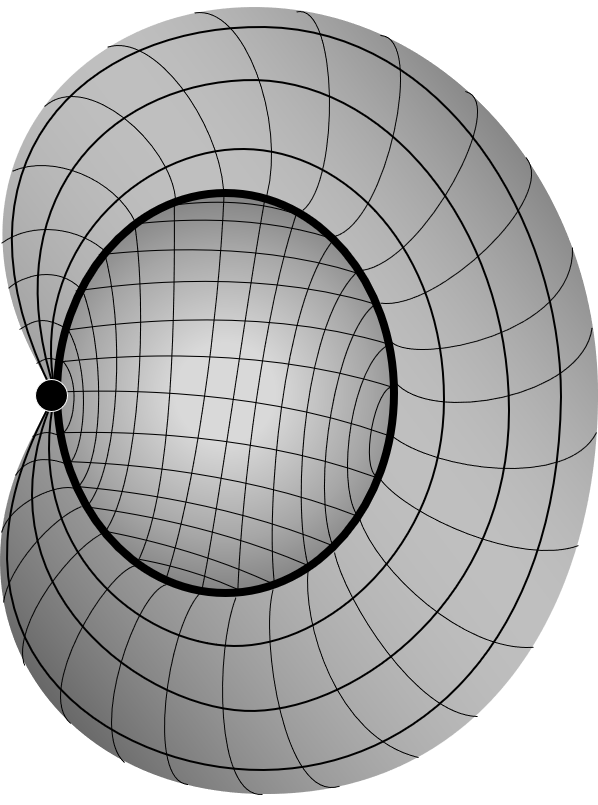}
\end{figure}

Any regular CW structure, such as the illustrated decomposition into little squares, constitutes a stratification of $\cY$ where $d$-dimensional strata are precisely the $d$-cells; passing to a subdivision further refines this stratification in the sense that every new cell is entirely contained in the interior of an old cell. On the other hand, one can discover a much coarser stratification by examining the topology of small neighborhoods around points of $\cY$. Up to homeomorphism, these fall into three different classes depending on whether the central point is at the pinch, on the singular equator, or on one of the manifold-like two-dimensional regions:
\begin{figure}[h!]
	\includegraphics[scale = 0.25]{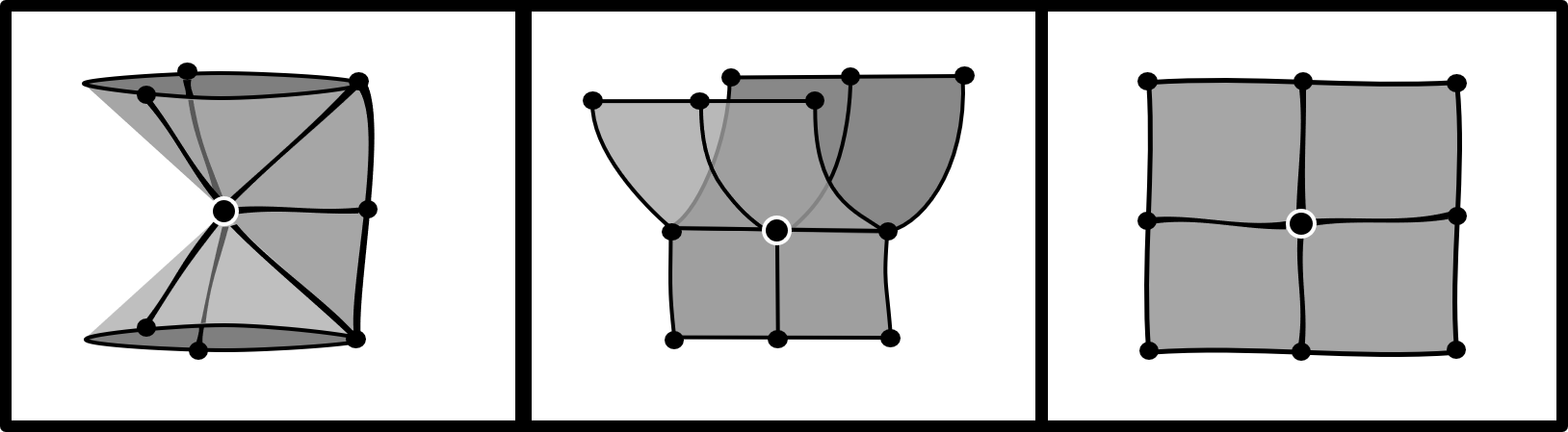}
\end{figure}

The neighborhoods above deformation-retract onto their central vertices and are therefore contractible; however, their one-point compactifications (obtained by collapsing their boundaries to points) are new stratified spaces with potentially interesting topology. The compactified neighborhood around the pinch-point is homeomorphic to two $2$-spheres joined at their north and south poles with a spanning disk across the middle. The compactified neighborhood of any point in the singular equator resembles a $2$-sphere whose interior has been partitioned in two by a disk. And finally, the compactified neighborhood around any point in either of the two-dimensional regions is homeomorphic to an ordinary $2$-sphere:  
\begin{figure}[h!]
	\includegraphics[scale = 0.25]{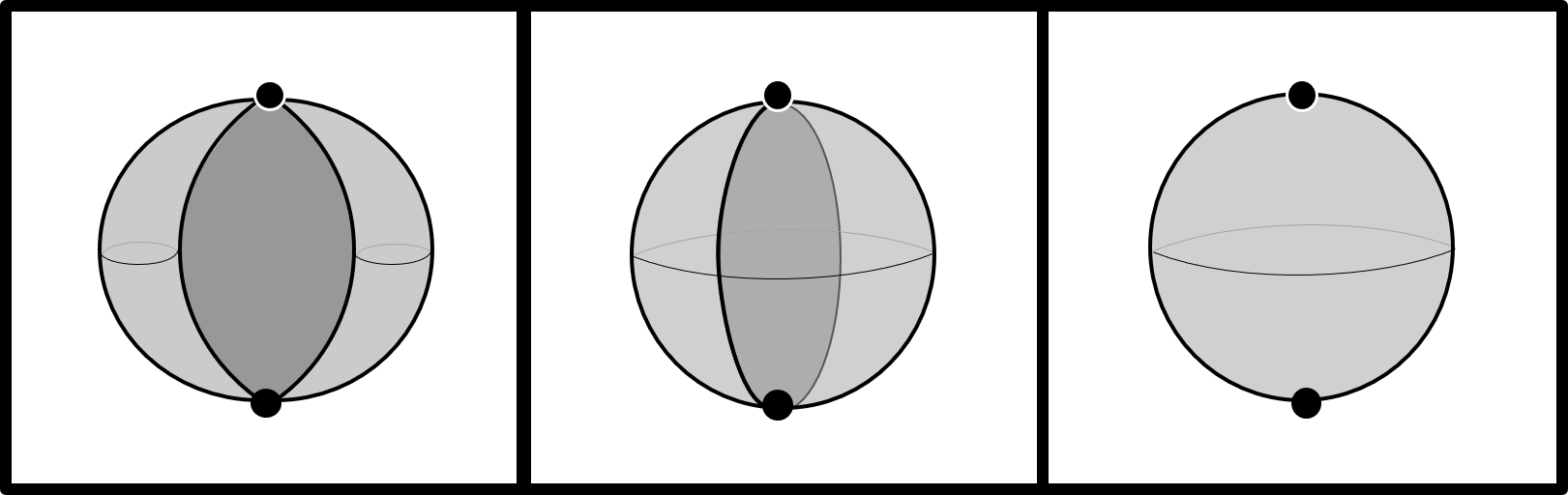}
\end{figure}

Our main result here involves algorithmically recovering the coarsest stratification of a finite regular CW complex where all strata are unions of cells. This is called the {\bf canonical stratification} of the complex; its existence and uniqueness for a special class of spaces (called  {\em pseudomanifolds}) plays a central role in Goresky and MacPherson's proof of the topological invariance of intersection homology \cite[Sec 4]{goresky:macpherson:83}. Our argument, much like theirs, has an intuitive geometric core but invokes algebraic and categorical machinery. For the purposes of this introductory section, we focus on geometry and ask: given a finite cellulation of $\cY$, how might one identify the canonical strata and determine which cells lie in each canonical stratum, as shown below?

\begin{figure}[h!]
	\includegraphics[scale = 0.3]{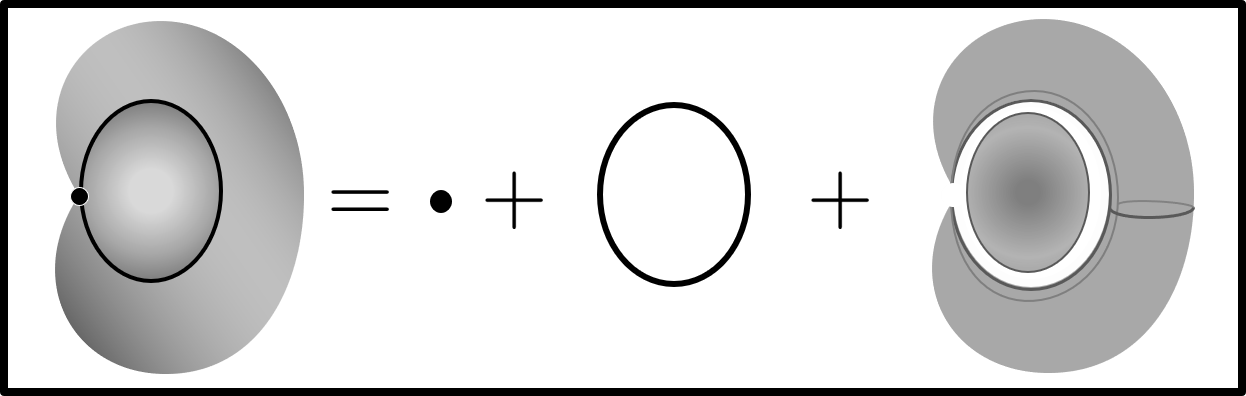}
\end{figure}

\noindent (Since the pinch point has a different compactified neighborhood than a generic point on the singular equator, it must constitute a separate stratum.)

In light of the discussion above, one hopes to recover canonical strata by clustering together cells  whenever they exhibit similar compactified neighborhoods. But already in this example, we encounter two significant difficulties: first, the compactified neighborhoods do not distinguish cells in the two $2$-strata from each other. The second difficulty is somewhat subtler --- although the compactified neighborhoods of cells in the $0$-stratum and $1$-stratum are not homeomorphic, both are homotopy-equivalent to a wedge of two $2$-spheres. Therefore, they can not be distinguished by weaker, more computable topological invariants such as cohomology. We tackle the first problem by constructing a {\em complex of cosheaves} which encodes how local topology varies across cells, and we bypass the second problem by working exclusively in the category of {\em cohomologically} stratified spaces. 

Consider the task of determining the canonical stratification of a finite-dimensional regular CW complex $\cX$ into $R$-cohomology manifolds, where $R$ is a fixed nontrivial commutative ring with unity.  In this case, our complex of cosheaves assumes a particularly appealing form: it is a functor \[\bL^\bullet:\Face\op(\cX) \to \Ch(R)\] from the poset of cells in $\cX$ (where $x > y$ denotes that $y$ is a face of $x$, or equivalently, that $x$ is a co-face of $y$) to the category of lower-bounded $R$-module cochain complexes. To each cell $x$ of $\cX$, it assigns a cochain complex $\bL^\bullet(x)$:
\[
\xymatrixcolsep{0.4in}
\xymatrix{
\bL^0(x) \ar@{->}[r]^{\beta^0_x}  & \bL^1(x) \ar@{->}[r]^{\beta^1_x}  & \bL^2(x) \ar@{->}[r]^{\beta^2_x} & \cdots,
}
\] 
where $\bL^d(x)$ is the free $R$-module generated by all $d$-dimensional co-faces of $x$ --- in particular, the module $\bL^d(x)$ is trivial for $d < \dim x$ and $d > \dim \cX$. The co-differentials $\beta^\bullet_x$ are inherited verbatim from incidence degrees among cells of $\cX$ taking values in $R$. Thus, $\bL^\bullet(x)$ computes the (reduced) cohomology of a compactified small open neighborhood around $x$ in $\cX$. Given another cell $y$ satisfying $x \geq y$, there is an inclusion map 
\[
\bL^\bullet(x \geq y):\bL^\bullet(x) \inj \bL^\bullet(y)
\] arising from the fact that co-faces of $x$ are also co-faces of $y$. Here is a simple version of our main result (the full statement has been recorded in Theorem \ref{thm:full}).

\begin{theorem}
\label{thm:main}
There is a category $\bS$ obtained from $\Face\op(\cX)$ by formally inverting a particular subset of those face relations $x \geq y$ for which $\bL^\bullet(x \geq y)$ induces  isomorphisms on cohomology; two cells lie in the same canonical stratum of $\cX$ if and only if they are isomorphic in $\bS$.
\end{theorem}

Before proceeding to the $\bL^\bullet$-induced definition of $\bS$ and a proof of this theorem, we highlight three salient features. First, the construction of $\bS$ is algorithmic and isomorphism classes of its objects are explicitly (and efficiently) computable. Second, the non-isomorphisms in $\bS$ reveal when a canonical stratum contains another in its boundary. And third, our construction actually yields a nested sequence of intermediate categories
\[
\bS_0 \inc \bS_1 \inc \cdots \inc \bS_{\dim \cX} = \bS,
\] each of which receives a functor from $\Face\op(\cX)$ compatible with the previous ones; and all the canonical strata of dimension exceeding $(\dim \cX - d)$ may be recovered from certain isomorphism classes in $\bS_d$.

This paper is organized as follows. Section \ref{sec:background} tersely collects relevant background material involving stratifications and cosheaves, while Section \ref{sec:lochom} describes the local cohomology complex of cosheaves $\bL^\bullet$ associated to a regular CW complex $\cX$. In Section \ref{sec:top} we use $\bL^\bullet$ to define an initial functor $\Face\op(\cX) \to \bS_0$ via categorical localization, and prove that the top-dimensional canonical strata of $\cX$ correspond bijectively with isomorphism classes of (images of) top-cells in $\bS_0$. Section \ref{sec:low} contains the heart of our argument: it concludes the proof of Theorem \ref{thm:main} by describing the inductive construction of $\Face\op(\cX) \to \bS_{d}$ from the previous $\Face\op(\cX) \to \bS_{d-1}$. The final Section \ref{sec:comp} may be read independently of the preceding ones --- it is devoted to computational matters, and includes an efficient distributed algorithm for recovering canonical strata in practice.

\subsection*{Context}

Our goal is to provide a principled topological pre-processor for various {\em manifold learning} algorithms. The fundamental aim of manifold learning is to automatically infer the intrinsic dimension of a  compact submanifold $M$ (often with boundary) of Euclidean space from a finite point sample $P \subset \R^n$ lying on, or at least near, $M$. An enormous amount of work has been done in this field, and we will make no attempt to even summarize it here --- the curious reader is referred to \cite{mfdlrn:08} and its myriad citations. In most cases, one constructs a graph around the points in $P$ by inserting all edges of Euclidean length bounded above by a parameter threshold $\epsilon > 0$. Shortest paths along this graph then serve as proxies for intrinsic geodesics of $M$ and are used to find distance-preserving projections to lower-dimensional subspaces of $\R^n$ via spectral methods. 

When $M$ is not a single manifold, but rather a union of different manifolds across several dimensions, our standard manifold learning methods necessarily result in over-fitting along regions where $M$ has low intrinsic dimension. Theorem \ref{thm:main} provides a remedy by conferring the ability to partition the Delaunay triangulation \cite[Ch 5]{compgeo}  around $P$ at radius $\epsilon$ into distinct clusters, each of which is guaranteed to lie on a cohomology manifold of known dimension\footnote{The probabilistic version of such a framework is described in \cite{bendich2}.}. These clusters may then be independently subjected to standard manifold learning techniques. Perhaps more important from a practical perspective is the fact that the computations involved in partitioning $P$ are easily distributed across several processors. And as a straightforward byproduct, one also gains the capacity to identify (those points of $P$ which lie in) every anomalous, singular region of $M$. 

There now exists a substantial body of work where filtered (co)homology groups of cell complexes built around data points play a central role \cite{carlsson:09}, \cite{ghrist:16}, \cite{edelsbrunner:harer:10}. With the knowledge of canonical strata of such cell complexes comes the prospect of efficiently computing far more refined topological invariants such as intersection cohomology \cite{goresky:macpherson:80} groups. Also accessible, thanks to suitable flavors of Morse theory \cite{nanda:15}, \cite{nanda:tamaki:tanaka:16}, is the ability to cluster canonical $d$-strata by $(d+1)$-dimensional cobordisms internal to the ambient dataset.

\subsection*{Acknowledgements}

I am indebted to Justin Curry, Rob Ghrist and Ulrike Tillmann for insightful discussions, and to Adam Brown and Yossi Bokor for carefully reading (and discovering oversights in) earlier drafts of this paper. I am also grateful to the two anonymous referees for their corrections, comments and suggestions for improvement. This work was partially supported by The Alan Turing Institute under the EPSRC grant EP/N510129/1 and by the Friends of the Institute for Advanced Study. It is dedicated to Mark Goresky and Bob MacPherson, whose efforts over four decades have built some of my favorite playgrounds.

\section{Preliminaries} \label{sec:background}

Our primary references for stratification theory are \cite{goresky:macpherson:88}, \cite{kirwan:88}, and \cite{weinberger:94}; for categorical localization see  \cite{dwyer:kan:80} and \cite{gabriel:zisman:67}; and for cellular (co)sheaves can be found in \cite{curry:13} and \cite{shepard:85}. The interplay between cell structures and general stratifications on the same underlying space has been explored thoroughly in \cite{tamaki:16}. Here we will be concerned entirely with cohomologically stratified spaces, which the reader may have encountered before in \cite[Sec 4.1]{goresky:macpherson:83}, \cite[Sec 3.3]{getz:goresky:12}, or \cite[Sec 5]{rourke:sanderson:99}.

\subsection{Cohomological stratifications}

Fix a non-trivial commutative ring $R$ with identity and recall that an {\bf $R$-cohomology manifold} of dimension $d \geq 0$ is a locally compact Hausdorff topological space $\cZ$ where each point $z$ has an open neighborhood $\cU_z \subset Z$ whose compactly-supported $R$-valued cohomology agrees with that of $d$-dimensional Euclidean space $\R^d$, i.e.,
\begin{align}
\label{eqn:cohmfd}
\HBM^\bullet\left(\cU_z; R\right) \simeq \begin{cases} R & \text{ if } \bullet = d, \\
																																			0 & \text{ otherwise. }
																								\end{cases}
\end{align}
It is essential to use compactly-supported (rather than ordinary singular) cohomology to detect $d$-dimensionality, since the usual cohomology of a $d$-dimensional open ball agrees with that of a point regardless of $d$. Readers unfamiliar with this notion may safely replace the compactly-supported cohomology of any open set encountered henceforth by the standard cohomology of its one-point compactification relative to the additional point:
\[
\HBM^\bullet\left(\cU_z; R\right) \simeq H^\bullet\left(\cU_z\cup\{\infty\},\{\infty\};R\right).
\]

\begin{definition}
\label{def:strat}
An $n$-dimensional {\bf (cohomological) stratification} of a locally compact Hausdorff topological space $\cX$ is an ascending sequence of closed subspaces
\[
\varnothing = \cX_{-1} \subset \cX_{0} \subset \cdots \subset \cX_{n-1} \subset \cX_n = \cX,
\] 
where the connected components of each difference $(\cX_d - \cX_{d-1})$, called $d$-dimensional {\bf strata}, are required to obey the following axioms.
\begin{itemize}  
\item {\bf Frontier:} if a stratum $\sigma$ intersects the closure of another stratum $\tau$, then in fact $\sigma$ is completely contained in the closure of $\tau$ and $\dim \tau \geq \dim \sigma$ (with equality occurring if and only if $\tau = \sigma$). This relation, denoted by $\tau \succeq \sigma$ henceforth, forms a partial order on the set of all strata.

\item {\bf Link:} for each $d$-stratum $\sigma$, there exists an $(n-d-1)$-dimensional stratified space $\cL = \cL(\sigma)$, called the {\em link} of $\sigma$: 
\[
\varnothing = \cL_{-1} \subset \cL_0 \subset \cdots \subset \cL_{n-d-1} = \cL,
\] so that every open neighborhood around a point $p$ in $\sigma$ admits a {\em basic} open sub-neighborhood $\cU_p \subset \cX$ with the following structure. The intersections $\cU_p \cap \cX_i$ are empty for $i < d$, while for $d \leq i \leq n$ there are stratified quasi-isomorphisms of compactly-supported singular cochain complexes (of $R$-modules):
\begin{align}
\label{eqn:unif}
\xymatrixcolsep{0.25in}
\xymatrix
{
\CBM^\bullet(\Cone\cL_{i-d-1} \times \R^d) \ar@{->}[r]^-{\simeq} & \CBM^\bullet(\cU_p \cap \cX_i), 
}
\end{align} 
where $\Cone\cL_\bullet$ denotes the open cone\footnote{The open cone on $\cZ$ is the quotient of $\cZ \times [0,1)$ by the relation which identifies $(z,0)$ with $(z',0)$ for any $z$ and $z'$ in $\cZ$; by convention, the open cone on the empty set is the one-point space.}  on $\cL_\bullet$. 
\end{itemize}
\end{definition}
By quasi-isomorphism here (and elsewhere) we mean a cochain map which induces isomorphisms on all cohomology groups. The additional requirement that the maps from (\ref{eqn:unif}) be stratified means precisely that they behave naturally in three separate ways: they should be {\em filtered}, {\em stratum-preserving} and {\em refinable} as explained below.
\begin{enumerate}
\item They should respect (contravariant maps induced by) the inclusions $\cL_\bullet \subset \cL_{\bullet+1}$ in their domains and $\cX_\bullet \subset \cX_{\bullet+1}$ their codomains so that the following squares commute:
\[
\xymatrixrowsep{0.2in}
\xymatrixcolsep{0.2in}
\xymatrix{
 \CBM^\bullet\left(\Cone\cL_{i-d-1} \times \R^d\right) \ar@{->}[r]^-{\simeq} &  \CBM^\bullet\left(\cU_p \cap \cX_{i}\right)  \\
  \CBM^\bullet\left(\Cone\cL_{i-d} \times \R^d\right)  \ar@{->}[u] \ar@{->}[r]_-{\simeq} & \CBM^\bullet\left(\cU_p \cap \cX_{i+1}\right) \ar@{->}[u].
}
\]
\item They should preserve strata in the sense that there exist surjective set-maps 
\[
\Phi_i: \set{(i-d-1)-\text{strata of } \cL} \surj \set{i-\text{strata } \tau \succeq \sigma \text{ of } \cX} \] so that for each $i$-stratum $\tau \succeq \sigma$, the cochain maps from (\ref{eqn:unif}) restrict to quasi-isomorphisms 
\[
\CBM^\bullet(\Cone \Phi^{-1}_i(\tau) \times \R^d) \stackrel{\simeq}{\longrightarrow} \CBM^\bullet(\cU_p \cap \tau).
\] 
\item They should refine to smaller basic neighborhoods $\cV_p$ so that the following triangle of cochain maps commutes:
\[
\xymatrixrowsep{0.2in}
\xymatrixcolsep{0.25in}
\xymatrix{
\CBM^\bullet(\Cone\cL_{i-d-1} \times \R^d) \ar@{->}[r]^-{\simeq} \ar@{->}[rd]_{\simeq} & \CBM^\bullet(\cU_p \cap \cX_i) \ar@{->}[d]^{}  \\
																																										& \CBM^\bullet(\cV_p \cap \cX_i) 
}
\]
Here the vertical map is induced by the inclusion $\cV_p \subset \cU_p$; and since the triangle commutes, this map is forced to also be a quasi-isomorphism. 
\end{enumerate}

The best way to acquire intuition for these constraints is to examine (pictures of) links associated to singular strata in low-dimensional spaces. 
\begin{example}
The decomposition of the space $\cY$ from the Introduction into the pinch point, the equatorial circle, a disk 2-stratum and a toral 2-stratum is a cohomological stratification. The pinch point lies at the frontier of the equatorial circle, which in turn lies at the frontier of both 2-strata. The link $\cL$ of the singular equator (which we will call $\sigma$ here) is a three-point space. An open cone on $\cL$ is therefore the union of three half-open intervals along a common boundary point. The product of this cone with the real line is homeomorphic to small neighborhoods around points in $\sigma$:

\begin{figure}[h!]
	\includegraphics[scale = 0.3]{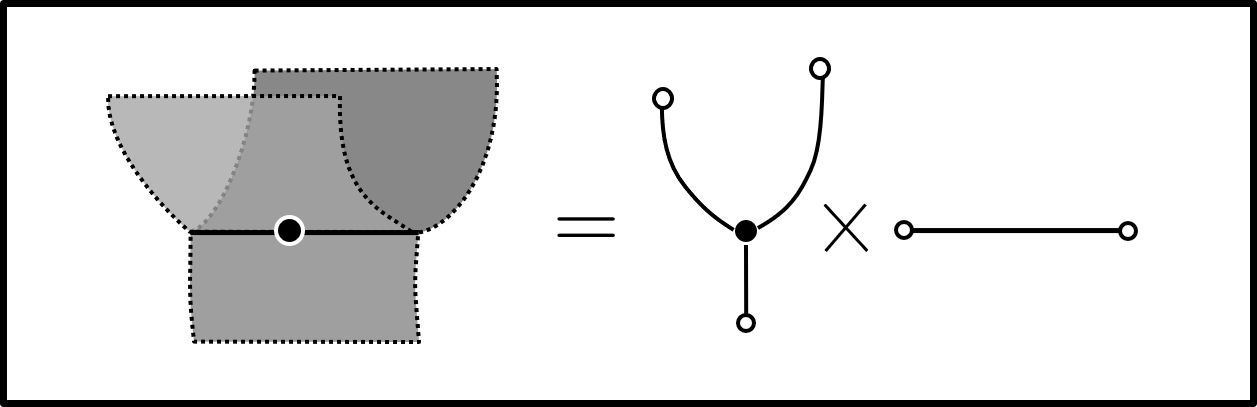}
\end{figure}

\noindent From such a homeomorphism, one obtains the desired cohomological equivalence (\ref{eqn:unif}). This equivalence is 
\begin{itemize}
	\item filtered, provided we stratify the open cone $\Cone\cL$ with one 0-stratum and three 1-strata (and the real line by a single 1-stratum), 
	\item stratum-preserving, provided that the map $\Phi_2$ sends two of the points in $\cL$ to the toral 2-stratum and the third point to the disk 2-stratum, and
	\item refinable, because one can shrink the neighborhood around $\sigma$ by a small amount while still preserving the homeomorphism. 
\end{itemize}
\end{example}

\begin{remark}
\label{ex:notcohstr}
To see a non-example of a cohomological stratification, try combining the pinch-point $p$ of $\cY$ with the singular equator $\sigma$ into a single 1-stratum. The refinability constraint will not be satisfied in this case. In particular, given any point $q \in \sigma$ near $p$ with two basic neighborhoods $V_q \subset U_q$ satisfying $p \in U_q - V_q$, the inclusion of $V_q$ into $U_q$ fails to induce an isomorphism on 2-dimensional compactly-supported cohomology. Instead, one obtains a rank one map $R^2 \to R^2$.
\end{remark}

The following observations concern technical aspects of Defintion \ref{def:strat}.
\begin{enumerate}
\item By a theorem of Eilenberg and Zilber \cite{eilenberg:zilber:53}, links are well-defined (only) up to filtered quasi-isomorphism type: we may replace $\cL$ by a filtered cochain complex by substituting a tensor product for the left side of (\ref{eqn:unif}).
\item The link axiom (when $i = d$) guarantees that each $d$-stratum is indeed a $d$-dimensional $R$-cohomology manifold.
\item This definition, unlike similar notions which typically appear in the intersection homology literature, does not require $\cX$ to be a {\bf pseudomanifold} (see \cite[Sec 1.1]{goresky:macpherson:83} for instance). In other words, $\cX_{n-1}$ need not equal $\cX_{n-2}$.
\end{enumerate}

 Every finite-dimensional regular\footnote{Regularity in this context means that the characteristic map of every cell is an embedding.} CW complex admits a cohomological stratification whose strata are the cells. 

\begin{definition}
\label{def:skelstrat}	
The {\bf skeletal stratification} of a finite-dimensional regular CW complex $\cX$ is defined as follows: writing the face partial order among cells as $\geq$, we recall that each cell $y$ of $\cX$ has
\begin{itemize}
\item an open star $\st(y)$ containing all cells $x$ which satisfy $x \geq y$, and 
\item a link $\lk(y)$, containing all cells $x$ that share a co-face but no face with $y$. 
\end{itemize}
The $d$-dimensional skeletal strata are the $d$-cells of $\cX$, and the frontier partial order coincides with $\geq$. Since $\st(y)$ of a $d$-cell $y$ is simultaneously homeomorphic to all sufficiently small neighborhoods around points in $y$ and to $\Cone\lk(y) \times \R^d$, the link $\cL(y)$ in the sense of Definition \ref{def:strat} is precisely $\lk(y)$. 
\end{definition}

We call one stratification a {\em coarsening} of another whenever each stratum of the former is a union of strata of the latter. All stratifications encountered henceforth will be coarsenings of the skeletal stratification for a fixed finite-dimensional regular CW complex $\cX$. 

\begin{definition}
\label{def:canstrat}
The {\bf canonical} stratification of a finite-dimensional regular CW complex $\cX$ is the coarsest stratification of $\cX$ whose strata are all unions of cells. 
\end{definition}

It may not be immediately clear why canonical stratifications should exist at all. We will establish both their existence and uniqueness for regular CW complexes in the sequel via an explicit construction, which can be made efficiently algorithmic whenever the number of cells in $\cX$ is finite. The core of our construction is heavily inspired by Goresky and MacPherson's proof that {\em canonical $\bar{p}$-filtrations} of stratified pseudomanifolds exist and are unique \cite[Sec 4.2]{goresky:macpherson:83}. 

\begin{remark}
\label{rem:gmcomp}
For readers familiar with the Goresky-MacPherson construction: there are two differences between their canonical $\bar{p}$-filtration proof and the argument that we will present in the sequel here. Namely,
\begin{enumerate}
	\item since our starting point is a regular CW complex rather than a stratified pseudomanifold, we have direct recourse to the combinatorics of cell incidences which determine the topology of the underlying space; and,
	\item here we use the sheaf of compactly-supported cellular cochains on the underlying space rather than the sheaf of intersection chains with respect to a perversity function $\bar{p}:\Z_{\geq 2} \to \Z$. 
\end{enumerate}
Aside from these modifications, our arguments below will follow theirs quite closely.
\end{remark}

The following result is a direct consequence of the frontier axiom from Definition \ref{def:strat} for cells lying in top strata.
\begin{proposition}
\label{prop:upclose}
Let $\cX$ be a regular CW complex of dimension $n$ equipped with any stratification coarser than its skeletal stratification. If a cell $y$ of $\cX$ lies in an $n$-dimensional stratum $\sigma$, then so must every cell $x$ which satisfies $x \geq y$. 
\end{proposition}
\begin{proof}
Let $\tau$ be the unique coarse stratum containing $x$. Since $y$ lies in the boundary of $x$ by assumption, the closure of $\tau$ intersects $\sigma$ non-trivially at $y$. The desired conclusion now follows from the axiom of the frontier in Definition \ref{def:strat} and the fact that there are no strata of dimension exceeding $n$, since $\dim \tau \geq \dim \sigma = n$ forces $\tau = \sigma$.
\end{proof}

Thus, membership of cells in top-dimensional canonical strata is {\em upward-closed} with respect to the face partial order.

\subsection{Localizations of posets} \label{sec:localization}

The localization of a poset $\bP$ about a subcollection $\Sigma$ of its order relations is the minimal category containing $\bP$ in which the relations of $\Sigma$ have been rendered invertible \cite[Ch I.1]{gabriel:zisman:67}. 
\begin{definition} 
\label{def:locpos}
Let $(\bP,\geq)$ be a poset and $\Sigma = \{(x_\bullet \geq y_\bullet)\}$ a subset of its relations which is closed\footnote{That is, if $(p \geq q)$ and $(q \geq r)$ both lie in $\Sigma$, then so does $(p \geq r)$.} under composition. The {\bf localization} of $\bP$ about $\Sigma$ is a category $\bP[\Sigma^{-1}]$ whose objects are precisely the elements of $\bP$, while morphisms from $p$ to $q$ are given by equivalence classes of finite (but arbitrarily long) $\Sigma$-{\em zigzags} of order relations in $P$
\[
p \geq y_0 \leq x_0 \geq \cdots \geq y_k \leq x_k \geq q, 
\]
where:
\begin{enumerate}
\item only relations in $\Sigma$ and equalities can point backwards (i.e., $\leq$),
\item composition is given by concatenation, and
\item the trivial zigzag $p = p$ represents the identity of each element $p$.
\end{enumerate}
The equivalence between $\Sigma$-zigzags is generated by the transitive closure of the following basic relations. Two such zigzags are related
\begin{itemize}
\item {\em horizontally} if one is obtained from the other by removal of intermediate equalities:
\begin{align*}
\left(\cdots  \leq x \geq y = y \geq x' \leq \cdots \right) &\sim \left(\cdots  \leq x \geq x' \leq \cdots\right), \\ 
\left(\cdots  \geq y \leq x = x \leq y' \geq \cdots \right) &\sim \left(\cdots  \geq y \leq y' \geq \cdots\right),  
\end{align*}
\item or {\em vertically}, if they form the rows in a grid:
\begin{alignat*}{11}
p~ & ~\geq~ &   ~y_0~  &  ~\leq~ &    ~x_0~   &   ~\geq~ & ~\cdots~ & ~\leq~ &    ~x_k~   &  ~\geq~ & ~q \\
\roteq~  &              & \geqdn~  &              &  \geqdn~ &                &                &             & \geqdn~    &               &   \roteq    \\
p~ & ~\geq~ &   ~y'_0~  &  ~\leq~ &     ~x'_0~  &  ~\geq~  &  ~\cdots~ & ~\leq~ &     ~x'_k~ &   ~\geq~ & ~q
\end{alignat*}
where all vertical face relations (also) lie in $\Sigma$.
\end{itemize} 
\end{definition}

The category $\bP[\Sigma^{-1}]$ enjoys the following universal property \cite[Lem I.1.3]{gabriel:zisman:67} --- there is a canonical {\bf localization functor} $\bP \to \bP[\Sigma^{-1}]$ which acts by inclusion. In particular, it sends each element of $\bP$ to itself and each relation $x \geq y$ to its own equivalence class of $\Sigma$-zigzags. Moreover, any functor $\bP \to \bC$ which maps order relations in $\Sigma$ to isomorphisms in the target category $\bC$ factorizes uniquely across the localization functor.

\subsection{Cellular cosheaves}
 
Sheaves (and their dual cosheaves) assign algebraic objects to open sets of topological spaces \cite{bredon:97}. Our interest here is in a particularly tame and computable sub-class of cosheaves, where the base space is always a finite-dimensional regular CW complex $\cX$ and the assigned algebraic objects are cell-wise constant. We write $\Face\op(\cX)$ to indicate the poset of cells with order relation $x > y$ denoting that $y$ is a face of $x$. 

\begin{definition}
\label{def:cellcosh}
A {\bf cellular cosheaf} over $\cX$ taking values in an Abelian category $\bA$ is a functor $\bF: \Face\op(\cX) \to \bA$. Thus, it assigns to each cell $x$ an $\bA$-object $\bF(x)$ and to each face relation $x \geq y$ an $\bA$-morphism $\bF(x \geq y): \bF(x) \to \bF(y)$ so that
\begin{itemize}
\item $\bF(x = x)$ is the identity on $\bF(x)$ for each cell $x$, and
\item $\bF(y \geq z) \circ \bF(x \geq y)$ equals $\bF(x \geq z)$ across any triple of cells $x \geq y \geq z$.
\end{itemize}
We call $\bF(x)$ the {\bf stalk} of $\bF$ at $x$, and call $\bF(x \geq y)$ the {\bf extension map} of $\bF$ at $x \geq y$.
\end{definition}

A {\bf morphism} of cellular cosheaves $\eta:\bF \to \bG$ over $\cX$ is a natural transformation; it assigns to each cell $x$ an $\bA$-morphism $\eta_x: \bF(x) \to \bG(x)$ so that for each $x \geq y$ the relevant square in $\bA$ commutes:
\[
\xymatrixrowsep{0.3in}
\xymatrixcolsep{0.42in}
\xymatrix{
\bF(x) \ar@{->}[r]^-{\bF(x \geq y)} \ar@{->}[d]_{\eta_x} & \bF(y) \ar@{->}[d]^{\eta_y} \\
\bG(x) \ar@{->}[r]_-{\bG(x \geq y)} & \bG(y)
}
\]
Cosheaf morphisms may be composed stalkwise, and there is always a zero morphism $0: \bF \to \bG$ which (as one might expect) assigns the zero map $0_x:\bF(x) \to \bG(x)$ in $\bA$ to each cell $x$ of $\cX$.
\begin{definition}
\label{def:compcosh}
A (lower-bounded) {\bf complex} of cellular cosheaves $\bF^\bullet$ (over $\cX$, taking values in $\bA$) is a sequence of $\bA$-valued cellular cosheaves over $\cX$ connected by cosheaf morphisms: 
\[
\xymatrixcolsep{0.3in}
\xymatrix{
\bF^0 \ar@{->}[r]^{\eta^0} & \bF^1 \ar@{->}[r]^{\eta^1} & \bF^2 \ar@{->}[r]^{\eta^2} & \cdots
}
\]
so that every successive composition $\eta^{\bullet+1}\circ\eta^\bullet$ yields the zero morphism.
\end{definition}
Every complex of cosheaves $\bF^\bullet$ on $\cX$ may be reinterpreted as a single cosheaf which takes values in the category $\Ch(\bA)$ of lower-bounded cochain complexes in $\bA$ --- consider the collection of commuting squares in $\bA$ that are assigned to face relations $x \geq y \geq z \geq \cdots$ of $\cX$:
\[
\xymatrixcolsep{0.35in}
\xymatrixrowsep{0.2in}
\xymatrix{
\bF^0(x) \ar@{->}[r] \ar@{->}[d] & \bF^1(x) \ar@{->}[r] \ar@{->}[d] & \bF^2(x) \ar@{->}[r] \ar@{->}[d] & \cdots \\
\bF^0(y) \ar@{->}[r] \ar@{->}[d] & \bF^1(y) \ar@{->}[r] \ar@{->}[d] & \bF^2(y) \ar@{->}[r] \ar@{->}[d] & \cdots \\
\bF^0(z) \ar@{->}[r]  \ar@{->}[d] & \bF^1(z) \ar@{->}[r] \ar@{->}[d] & \bF^2(z) \ar@{->}[r]  \ar@{->}[d] & \cdots \\
\vdots & \vdots & \vdots & \ddots
}
\]
The horizontal slices of this diagram confirm that a cochain complex in $\bA$ is allocated to each cell, while vertical arrows between these slices prescribe the desired extension maps. Taking cohomology horizontally, one therefore obtains an ordinary cellular cosheaf of (graded) $\bA$-objects over $\cX$. We call this the {\bf cohomology} of $\bF^\bullet$ and denote it by $H^\bullet\bF$.

\section{Local cohomology of CW complexes} \label{sec:lochom}

Let $\cX$ be a finite-dimensional regular CW complex with face poset $\Face\op(\cX)$ and $R$ a fixed nontrivial commutative ring with identity $1_R$. We write $\Mod(R)$ to denote the category of $R$-modules, and $\Ch(R)$ -- rather than the cumbersome $\Ch(\Mod(R))$ -- to indicate the category of cochain complexes of $R$-modules indexed by the non-negative natural numbers.
\begin{definition}{\cite[Def 7.1]{curry:16}}
\label{def:loccoh}
The {\bf local cohomology} $\bL^\bullet$ of $\cX$ is a complex of cosheaves
\[
\xymatrixcolsep{0.33in}
\xymatrix{
\bL^0 \ar@{->}[r]^{\beta^0} & \bL^1 \ar@{->}[r]^{\beta^1} & \bL^2 \ar@{->}[r]^{\beta^2} & \cdots
}
\]
over $\cX$ taking values in $\Mod(R)$, prescribed by the following data. 
\begin{enumerate} 
\item For each dimension $d \geq 0$ and cell $x$ of $\cX$, the cosheaf $\bL^d$ has as its stalk $\bL^d(x)$ the free $R$-module with basis
\[
\set{z \in \Face\op(\cX) \mid z \geq x \text{ and } \dim z = d}.
\] When $x \geq y$, the extension $\bL^d(x \geq y):\bL^d(x) \inj \bL^d(y)$ is determined by the obvious inclusion of basis cells. 
\item The cosheaf morphism $\beta^d$ assigns to each cell $x$ the map $\beta^d_x:\bL^d(x) \to \bL^{d+1}(x)$ defined by (linearly extending) the following action on basis cells. For each $d$-cell $z \geq x$, we have 
\[
\beta^d_x(z) = \sum_{w} \ip{w,z}_R \cdot w,
\]
where the sum is taken over all $(d+1)$-cells $w \geq x$, and $\ip{w,z}_R$ is the $R$-valued degree of the attaching map in $\cX$ from the boundary of $w$ onto $z$. (Since we have assumed that $\cX$ is regular, this number $\ip{w,z}_R$ takes values in $\{0,\pm 1_R\}$ for all cells $w$ and $z$.)
\end{enumerate}
\end{definition}

Routine calculations confirm that $\beta^{d+1} \circ \beta^d$ is always zero, and that the $\beta^\bullet_x$'s commute with all relevant extension maps. In light of the discussion following Defintion \ref{def:compcosh}, we will shift perspective and regard $\bL^\bullet$ as a single cosheaf valued in $\Ch(R)$. In this setting, the stalk $\bL^\bullet(x)$ lying over each cell $x$ is the entire cochain complex
\[
\xymatrixcolsep{0.33in}
\xymatrix{
\bL^0(x) \ar@{->}[r]^{\beta^0_x}  & \bL^1(x) \ar@{->}[r]^{\beta^1_x}  & \bL^2(x) \ar@{->}[r]^{\beta^2_x} & \cdots,
}
\] 
and $H^\bullet \bL(x)$ is the compactly-supported cohomology of $x$'s open star in $\cX$. By excision, one may describe $H^\bullet\bL(x)$ as the ordinary relative cohomology of the pair $\left(\overline{\st}(x),\partial\overline{\st}(x)\right)$ where $\overline{\st}(x)$ is the closure of $x$'s open star in $\cX$ and $\partial$ denotes the topological boundary:
\[
\partial\overline{\st}(x) = \overline{\st}(x)-\st(x).
\]
When $y$ is a zero-dimensional cell, we have $\partial\overline{\st}(y) = \lk(y)$.

\begin{definition}
	For each $d \geq 0$, we write $R[d]^\bullet$ to indicate the special cochain complex in $\Ch(R)$ which is trivial except for a single copy of $R$ in the $d$-th position:
\[
0 \to \cdots \to 0 \to R \to 0 \to \cdots
\] 
Thus, $H^\bullet R[d]$ is the cohomology of small neighborhoods in $d$-dimensional $R$-cohomology manifolds, as in (\ref{eqn:cohmfd}). 
\end{definition}

By Definition \ref{def:strat}, if a cell $x$ lies in some top-dimensional stratum of $\cX$, then we must have an isomorphism $H^\bullet\bL(x) \simeq H^\bullet R[\dim \cX]$. For cells of high dimension, this requirement has strong consequences.

\begin{proposition}
\label{prop:topcell}
If $n = \dim \cX$, then every $n$-cell $x$ has $\bL^\bullet(x) = R[n]^\bullet$. And moreover, every $(n-1)$-cell $y$ with $H^\bullet\bL(y) \simeq H^\bullet R[n]$ lies in the boundary of exactly two $n$-cells.
\end{proposition}
\begin{proof}
The first assertion follows from Definition \ref{def:loccoh} and the fact that there are no cells of dimension $\geq (n+1)$ in $\cX$. Turning to the second assumption, note that the local cohomology of an $(n-1)$-cell $y$ is only supported in dimensions $(n-1)$ and $n$:
\[
\bL^\bullet(y) = 0 \to \cdots \to 0 \to R \stackrel{\beta^{n-1}_y}{\longrightarrow} R^k \to 0 \to \cdots
\]
where $k$ counts the number of $n$-cells $x$ satisfying $x > y$. If $\bL^\bullet(y)$ is quasi-isomorphic to $R[n]^\bullet$, then $\beta^{n-1}_y$ must be injective (since $H^{n-1}\bL(y)$ is trivial), so $k > 0$. By definition, $\beta^{n-1}_y$ acts by sending a basis element of $R$ to a linear combination of basis elements in $R^k$; and by the regularity of $\cX$, all coefficients in this linear combination lie in $\{\pm 1_R\}$. The $n$-th cohomology of $\bL^\bullet(y)$ is the quotient
\[
H^n\bL(y) = \frac{R^k}{\text{im}(\beta^{n-1}_y)},
\] and it must be isomorphic to $R = H^n R[n]$ by assumption. Since $R$ is a commutative ring, it satisfies the invariant basis property --- that is to say, $R^k$ and $R^\ell$ are isomorphic as $R$-modules if and only if $k = \ell$ \cite[Def 1.1]{weibel:13}. Thus, we have $k = 2$ as desired.
\end{proof}

In light of the preceding result and (\ref{eqn:cohmfd}), one might hope to identify top-dimensional canonical strata by simply clustering together adjacent cells $x \geq y$ in $\cX$ whenever there exists a top-dimensional cell $w$ satisfying $w \geq x$ so that all extension maps below are quasi-isomorphisms:
\[
R[\dim \cX]^\bullet = \bL^\bullet(w) \stackrel{\simeq}{\inc} \bL^\bullet(x) \stackrel{\simeq}{\inc} \bL^\bullet(y).
\] 
Our next example shows that this (necessary) condition for $x$ and $y$ to belong to the same top stratum is not sufficient when $\dim w - \dim y > 1$.

\begin{example}
\label{ex:upclose}
Consider the $2$-dimensional simplicial complex depicted below: it is obtained by subdividing a parallelogram into four $2$-simplices along diagonals, and attaching an extra $2$-simplex along one of the four resulting half-diagonals.

\begin{figure}[h!]
\includegraphics[scale = 0.3]{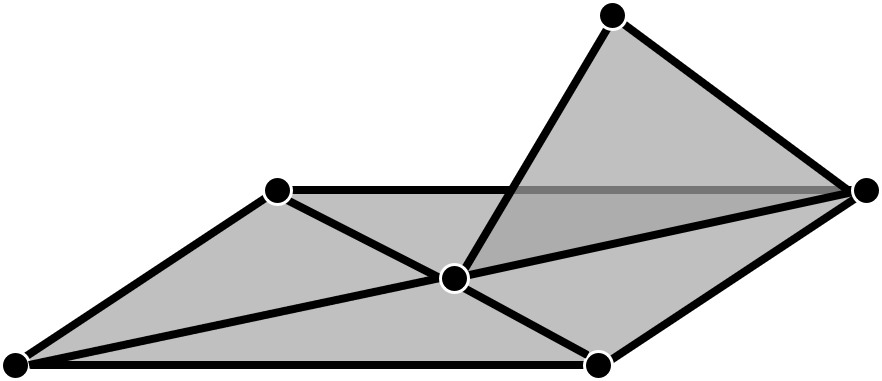}
\end{figure}

Let $w, x$ and $y$ respectively denote the extra $2$-simplex, the shared half-diagonal, and the central vertex. It is easy to confirm that $\bL^\bullet(y)$ is quasi-isomorphic to $R[2]^\bullet$ since its link has the cohomology of a circle. And moreover, given any sequence $w' > x' > y$ of cells where $w'$ and $x'$ differ from $w$ and $x$ respectively, the inclusion maps 
\[
R[2]^\bullet = \bL^\bullet(w') \inc \bL^\bullet(x') \inc \bL^\bullet(y)
\] are all quasi-isomorphisms. Note, however, that $\bL^\bullet(x)$ is not quasi-isomorphic to $R[2]^\bullet$ by Proposition \ref{prop:topcell} since there are three $2$-simplices containing $x$ as a face. Thus, $x$ does not belong to any canonical $2$-stratum, and hence (by Proposition \ref{prop:upclose}), neither does $y$.
\end{example}

\section{Extracting top strata} \label{sec:top}

Throughout this section, $\cX$ denotes an $n$-dimensional regular CW complex, $R$ is a fixed nontrivial commutative ring with identity, and $\bL^\bullet:\Face\op(\cX) \to \Ch(R)$ is the local cohomology complex of $\cX$ from Definition \ref{def:loccoh}. Consider the collection of face relations $(x \geq y)$ in $\Face\op(\cX)$ sent by $\bL^\bullet$ to quasi-isomorphisms, i.e.,
\begin{align}
\label{eqn:E0}
E_0 = \set{(x \geq y) \text{ in } \Face\op(\cX) \mid H^\bullet\bL(x \geq y) \text{ is an isomorphism}}
\end{align}
and define the subcollection
\begin{align}
\label{eqn:W0}
W_0 = \big\{(x \geq y) \text{ in } E_0 \mid & ~ H^\bullet\bL(y) \simeq H^\bullet R[n], \text{ and } \nonumber \\
																																	& (x' \geq y) \in E_0 \text{ for all } x' \geq y \text{ in }\Face\op(\cX) \big\}.
\end{align}
In other words, $W_0$ is what remains of $E_0$ when we impose 
\begin{itemize}
\item {\bf dimensionality}: retain only those face relations in which both cells have the local cohomology (isomorphic to that) of a top-dimensional cell, and 
\item {\bf upward-closure}: remove $(x \geq y)$ if there exists some $x' \geq y$ for which $\bL^\bullet(x' \geq y)$ is not a quasi-isomorphism (recall that cells lying in top strata are upward-closed by Proposition \ref{prop:upclose}).
\end{itemize}

It is clear that if $W_0$ contains both $(x \geq y)$ and $(y \geq z)$, then it also contains $(x \geq z)$.
\begin{definition}
\label{def:zerloc}
The category $\bS_0$ is the localization of the face poset $\Face\op(\cX)$ about $W_0$.
\end{definition}
Recall that $\bS_0$ has the cells of $\cX$ as objects, while its morphisms are equivalence classes of $W_0$-zigzags (as described in Definition \ref{def:locpos}). And, there is a canonical functor $\Face\op(\cX) \to \bS_0$ which is universal with respect to rendering all the face relations from $W_0$ invertible.

\begin{proposition}
\label{prop:lociso}
Two cells $w \neq z$ are isomorphic in $\bS_0$ if and only if there is a $W_0$-zigzag
\[
w \geq y_0 \leq x_0 \geq \cdots \geq y_k \leq x_k \geq z
\]
whose last relation $(x_k \geq z)$ lies in $W_0$. 
\end{proposition}
\begin{proof}
Given any $W_0$-zigzag from $w$ to $z$ such as the one depicted above, recall that every $(x_\bullet \geq y_\bullet)$ is in $W_0$. And by the upward-closure requirement, each preceding $(x_{\bullet-1} \geq y_\bullet)$ is automatically in $W_0$ (with the understanding that $x_{-1} = w$). Since the last relation $(x_k \geq z)$ is assumed to lie in $W_0$ as well, one can directly construct an inverse zigzag in $\bS_0$:
\[
z = z \leq x_k \geq y_k \leq \cdots \leq x_0 \geq y_0 \leq w = w.
\]
Conversely, if $w$ and $z$ are isomorphic in $\bS_0$, then there exist $W_0$-zigzags from $w$ to $z$ and vice-versa whose composites are (equivalent to) identity morphisms. We label some relevant cells in these zigzags:
\[
w \geq \cdots \leq x \geq z \quad \text{ and } \quad z \geq y \leq \cdots \geq w.
\]A brief examination of the horizontal and vertical relations from Definition \ref{def:locpos} confirms that any zigzag in the equivalence class of the identity must have all its forward-pointing $\geq$'s lying in $W_0$. As a consequence, when we compose the zigzags labeled above, we obtain $(x \geq y) \in W_0$. An appeal to the upward-closure requirement now forces $(x \geq z) \in W_0$, as desired. 
\end{proof}

Here is the main result of this section.
\begin{proposition}
\label{prop:topstrat}
Two $n$-dimensional cells lie in the same canonical $n$-stratum of $\cX$ if and only if they are isomorphic in $\bS_0$.
\end{proposition}
\begin{proof}
To show that the first assertion implies the second, let $\sigma$ be the $n$-stratum containing two distinct cells $w$ and $z$. Since $\sigma$ is connected by Definition \ref{def:strat}, there exists a zigzag of cells lying entirely in $\sigma$ from $w$ to $z$, say
\[
w \geq v_0 \leq v_1 \geq \cdots \geq v_k \leq z.
\]
By the same definition and Proposition \ref{prop:upclose}, it follows that every face relation appearing above lies in $W_0$. By adding an identity to the end of this zigzag, we may therefore construct a morphism in $\bS_0$ from $w$ to $z$,
\[
w \geq v_0 \leq v_1 \geq \cdots \geq v_k \leq z = z.
\]
It is now evident from Proposition \ref{prop:lociso} that $w$ and $z$ are isomorphic in $\bS_0$. In order to show the reverse implication, assume the existence of an isomorphism from $w$ to $z$ in $\bS_0$, say
\[
w \geq y_0 \leq x_0 \geq \cdots \geq y_k \leq x_k = z,
\]
where the last equality follows from the fact that $\cX$ contains no cells of dimension $> n$. Define the union 
\[
\cZ = \bigcup_{i=0}^k \textbf{st}(y_i)
\] of open stars, and note that $\cZ$ is connected --- every constituent $\textbf{st}(y_i)$ is connected and intersects the subsequent $\textbf{st}(y_{i+1})$ non-trivially (since both stars must contain $x_i$). By the dimensionality requirement on $W_0$, every point in $\cZ$ has an open neighborhood with compactly-supported cohomology isomorphic to $H^\bullet R[n] = \HBM^\bullet(\R^n)$, and by the upward-closure requirement no cell lying in $\cZ$ may lie in the boundary of a cell lying outside $\cZ$. Thus, the frontier and link requirements on strata imposed by Definition \ref{def:strat} are satisfied by the connected set $\cZ \subset \cX$, whence all cells lying in $\cZ$ must belong to the same canonical $n$-stratum $\sigma$ of $\cX$. Thus, $\sigma$ contains both $w$ and $z$ as desired.
\end{proof}

Finally, note that if an arbitrary cell $y$ lies in a (not necessarily canonical) $n$-stratum $\sigma$ of $\cX$, then there must exist an $n$-cell $w \geq y$ also lying in $\sigma$ by Proposition \ref{prop:upclose}, otherwise we arrive at the contradiction $H^n\bL(y) = 0$. Thus, we have the following consequence of Proposition \ref{prop:topstrat}.
\begin{corollary}
\label{cor:topstrat}
The canonical $n$-strata of $\cX$ correspond bijectively with isomorphism classes of its $n$-cells in $\bS_0$.
\end{corollary}

\section{Uncovering lower strata} \label{sec:low}

As before, let $\cX$ be the $n$-dimensional regular CW complex whose canonical strata we wish to discover. In this section, we will inductively construct the desired stratification  
\begin{align}
\label{eqn:Yd}
\cX = \cY_0 \supset \cY_1 \supset \cdots \supset \cY_n \supset \cY_{n+1} = \varnothing,
\end{align}
where each $\cY_\bullet \subset \cX$ is a regular CW subcomplex, so that the connected components of each difference $\cY_d - \cY_{d+1}$ are precisely the canonical $(n-d)$-strata of $\cX$. We have numbered the $\cY_\bullet$ in reverse (compare Definition \ref{def:strat}) merely to indicate the order in which they are constructed. In addition to the initial complex $\cY_0 = \cX$, the first step of our construction also requires using the sets $E_0$ and $W_0$ which were defined in (\ref{eqn:E0}) and (\ref{eqn:W0}) respectively.

\begin{definition}
\label{def:indtrip}
Given the triple $(\cY_{d-1}, E_{d-1}, W_{d-1})$, the $(n-d)$-dimensional regular CW subcomplex $\cY_{d}$ consists of all cells $y$ in $\cY_{d-1}$ for which there exists some cell $x \geq y$ in $\cY_{d-1}$ so that $(x \geq y)$ is not in the set $W_{d-1}$. (By upward closure, we may as well choose $x$ to be $y$.) Letting $\bL^\bullet_{d}:\Face\op(\cY_{d}) \to \Ch(R)$ be the local cohomology of $\cY_{d}$, define the set of face relations
\begin{align}
\label{eqn:Ed}
E_{d} = E_{d-1} \cap \set{(x \geq y) \text{ in } \Face\op(\cY_{d}) \mid H^\bullet\bL_{d}(x \geq y) \text{ is an isomorphism}},
\end{align}
and also
\begin{align}
\label{eqn:Wd}
W_{d} = W_{d-1} \cup \big\{(x \geq y) \text{ in } E_d \mid & ~ H^\bullet\bL_{d}(y) \simeq H^\bullet R[n-d], \text{ and } \nonumber \\
									                                     & (x' \geq y) \in E_{d} \text{ for all } x' \geq y \text{ in } \Face\op(\cY_{d})\big\}.
\end{align}
This produces the subsequent triple $(\cY_{d}, E_{d}, W_{d})$.
\end{definition}

Note that $W_\bullet$ is an increasing sequence of sets while both $\cY_\bullet$ and $E_\bullet$ are decreasing. 

\begin{example} To motivate the sets which appear in Definition \ref{def:indtrip}, consider the $2$-dimensional space $\cY_0$ pictured below: it is the closed cone over a wedge of two circles.
\begin{figure}[h!]
\includegraphics[scale = 0.35]{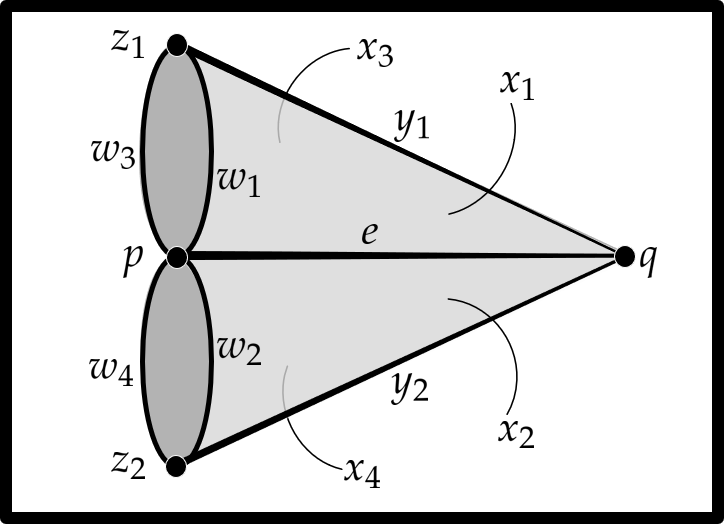}
\end{figure}

\noindent We have equipped $\cY_0$ with a regular CW decomposition consisting of four vertices ($p,q,z_*$), seven edges ($w_*, y_*, e$) and four 2-cells (all $x_*$). For convenience, let us assume that the coefficient ring $R$ is a field, so that cohomology computations reduce to knowledge of ranks of coboundary matrices. The stalk of the local cohomology cosheaf $\bL^\bullet_0$ over each of the four 2-cells is just $\bL_0^\bullet(x_*) = R[2]^\bullet$; and since each $y_*$ has two 2-cells in its open star,  we obtain the following stalks:
\[
\bL_0^\bullet(y_*) = (0 \to R \to R^2 \to 0 \to \cdots),
\]
where the only non-trivial map has rank one. It is straightforward to confirm that the based inclusions $\bL_0^\bullet(x_* \geq y_*)$ are all quasi-isomorphisms. Thus, the four face relations in $\cY_0$ of the form $x_* \geq y_*$ will be elements of $W_0$. All the other cells have the wrong local cohomology and will therefore not participate in $W_0$-relations: for instance,
\[
\bL_0^\bullet(w_*) = (0 \to R \to R \to 0 \to \cdots), 
\]
where again the nonzero map has rank 1, so the local cohomology is trivial and does not coincide with that of $R[2]^\bullet$. By upward closure, this also shows that neither the $z_*$ nor $p$ can participate in face relations of $W_0$. Similarly, the local cohomology of $e$ does not agree with that of $R[2]^\bullet$; the stalk is
\[
\bL_0^\bullet(e) = (0 \to R \to R^4 \to 0 \to \cdots),
\]
again with a rank one map. This disqualifies $q$ from participating in a face relation that lies in $W_0$. 
On the other hand, all four $(w_* \geq p)$ lie in $E_0$, so the cosheaf $\bL^\bullet_0$ is unable to distinguish generic cells lying on the two circles to the left from the singular point $p$. But once we remove all those $x_*$ and $y_*$ cells which are involved in face relations of $W_0$, we obtain a $1$-dimensional CW subcomplex $\cY_1 \subset \cY_0$ whose local cohomology correctly distinguishes $p$ from the $w_*$ and $z_*$:

\begin{figure}[h!]
\includegraphics[scale = 0.35]{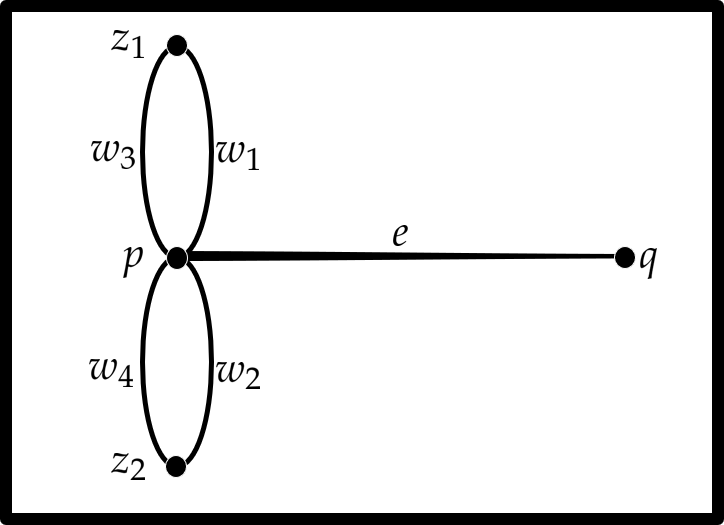}
\end{figure}

\noindent In particular, the local cohomology cosheaf $\bL^\bullet_1$ of $\cY_1$ has the following stalk over $p$:
\[
\bL_1^\bullet(p) = (0 \to R \to R^5 \to 0 \to \cdots),
\]
with the non-trivial map having rank 1; thus $\bL_1^\bullet(p)$ manifestly has different cohomology than $R[1]^\bullet = \bL_1^\bullet(w_*) \simeq \bL_1^\bullet(z_*)$. Similar top-down constructions (which measure the failure of $d$-dimensionality for decreasing $d$) have been employed since the very inception of stratification theory, dating back to the original work of Whitney \cite{whitney:65}.
\end{example}

\begin{remark}	
Two properties of the sequence $\cY_\bullet$ mentioned in Definition \ref{def:indtrip} may be inductively justified as follows, using the fact that $\cY_0$ is an $n$-dimensional regular CW complex as a base case. 
\begin{enumerate}

\item To see that $\cY_d \subset \cY_{d-1}$ is a regular CW subcomplex, note that if a cell $y$ lies in the difference $\cY_{d-1} - \cY_{d}$, then we have $(y = y) \in W_{d-1}$. And given any cell $x \geq y$ in $\cY_{d-1}$, we have $(x \geq y) \in W_{d-1}$ by upward closure, whence $(x = x) \in W_{d-1}$. Thus, $x$ also lies in the difference $\cY_{d-1} - \cY_{d}$. Since the collection of cells removed from $\cY_{d-1}$ to obtain $\cY_d$ is upward-closed with respect to the face partial order, $\cY_d$ is a regular CW subcomplex of $\cY_{d-1}$.

\item To see that $\dim \cY_d \leq (n-d)$, assume $\dim \cY_{d-1} \leq (n-d+1)$. But now, for each cell $w$ in $\cY_{d-1}$ of dimension $(n-d+1)$, we have $(w = w) \in W_{d-1}$ by Proposition \ref{prop:topcell}, so $w$ is not in $\cY_d$. Since $\cY_d$ contains no cells of dimension $(n-d+1)$, its dimension can not exceed $(n-d)$.
\end{enumerate}
\end{remark}

\begin{definition}
\label{def:scat}
For each $d > 0$ in $[n]$, the category $\bS_d$ is the localization of $\Face\op(\cX)$ about the set $W_d$ from (\ref{eqn:Wd}).
\end{definition}
Recall from Section \ref{sec:localization} that there is a universal functor $\Face\op(\cX) \to \bS_d$ for each $d$; and since $W_{d-1} \subset W_{d}$ holds by (\ref{eqn:Wd}), the universal property of localization guarantees a unique functor $\bS_{d-1} \to \bS_{d}$ which makes the following diagram commute:
\[
\xymatrixrowsep{0.25in}
\xymatrixcolsep{0.25in}
\xymatrix{
\Face\op(\cX) \ar@{->}[d] \ar@{->}[dr]& \\
\bS_{d-1} \ar@{->}[r] & \bS_d
}
\]
In fact, $\bS_{d-1} \to \bS_d$ admits a convenient explicit description: note that every $W_{d-1}$-zigzag is automatically a $W_d$-zigzag, so $\bS_{d-1}$ is a subcategory of $\bS_d$ and the  inclusion $\bS_{d-1} \inc \bS_d$ satisfies this universal property. Our next result establishes a certain $W_\bullet$-induced monotonicity for morphisms across the entire nested sequence $\bS_\bullet$; in its statement and beyond, $[p]$ denotes the set $\set{0,\ldots,p}$ for each integer $p \geq 0$.
\begin{lemma}
\label{lem:acyclicity}
Fix $d$ in $[n]$ along with cells $w$ and $z$ in $\cX$. For each $W_d$-zigzag $\zeta$ from $w$ to $z$, say
\[
\zeta = (w \geq y_0 \leq x_0 \geq \cdots \geq y_k \leq x_k \geq z),
\]
define the function $\varphi_\zeta:[k] \to [d]$ by setting
\begin{align}
\label{eqn:phi}
\varphi_\zeta(p) = \min\set{q \geq 0 \mid (x_p \geq y_p) \in W_q}.
\end{align}
Then, $\varphi_\zeta$ is monotone increasing in the sense that $\varphi_\zeta(p-1) \leq \varphi_\zeta(p)$ for all $p \geq 1$ in $[k]$.
\end{lemma}
\begin{proof}
Set $\varphi_\zeta(p) = q$ for a fixed $p$, so that $(x_{p} \geq y_{p}) \in W_q - W_{q-1}$, whence both $x_p$ and $y_p$ are cells in $\cY_q - \cY_{q+1}$. Since $x_{p-1} \geq y_p$ holds in $\cX$, the cell $x_{p-1}$ must lie in $\cY_r - \cY_{r+1}$ for some $r \leq q$ because $\cY_\bullet$ is a decreasing sequence of CW subcomplexes of $\cX$. Thus, $(x_{p-1} \geq y_{p-1}) \in W_r$, so we have
\[
\varphi_\zeta(p-1) \leq r \leq q = \varphi_\zeta(p),
\] as desired. 
\end{proof}

The cells of $\cX$ are ordered by dimension in the sense that the poset $\Face(\cX)$ is only allowed to have morphisms from higher-dimensional to lower-dimensional cells. Since morphisms in $\bS_\bullet$ are zigzags, one can not expect such a dimensional monotonicity to hold verbatim. But the $W_\bullet$-induced monotonicity from the previous lemma precludes the existence of certain morphisms in  $\bS_\bullet$.
\begin{corollary}
\label{cor:acyclicity}
Given $i \leq j$ in $[n]$ along with cells $w$ in $\cY_j$ and $z$ in $\cY_i - \cY_j$, there are no morphisms in $\bS_d$ from $w$ to $z$ for any $d$ in $[n]$.
\end{corollary}
\begin{proof}
Note that a direct face relation $w \geq z$ can not hold in $\cX$ since $\cY_j$ is a CW subcomplex of $\cY_i$. Thus, any morphism in $\bS_d$ from $w$ to $z$ admits a genuine zigzag representative. Proceeding by contradiction, assume the existence of such a zigzag $\zeta$:
\[
w \geq y_0 \leq x_0 \geq \cdots \geq y_k \leq x_k \geq z.
\]
The face relations $w \geq y_0$ and $x_k \geq z$ in $\cX$ guarantee $j \leq \varphi_\zeta(0)$ and $\varphi_\zeta(k) \leq i$, where $\varphi_\zeta$ is the function from (\ref{eqn:phi}). The monotonicity of $\varphi_\zeta$ from Lemma \ref{lem:acyclicity} now forces $j \leq i$, and hence, $i = j$. The cell $z$ is thus constrained to lie in $\cY_i - \cY_j = \varnothing$, which yields the desired contradiction. 
\end{proof}

As a direct consequence of the preceding result, isomorphism classes in $\bS_i$ of $(n-i)$-dimensional cells from $\cY_i$ remain unchanged across all the inclusions $\bS_i \inc \bS_j$ for $i \leq j$. The next result gives convenient alternate descriptions of such isomorphism classes across the entire sequence $\bS_\bullet$.
\begin{proposition}
\label{prop:isocon}
For each $d \in [n]$, the following are equivalent. Two cells $w$ and $z$ in $\cX$
\begin{enumerate}
\item lie in the same connected component of $\cY_d - \cY_{d+1}$,
\item are connected by a $W_d$-zigzag, whose relations ($\leq$ and $\geq$) all lie in $W_d - W_{d-1}$,
\item are isomorphic in $\bS_d$ to a common $(n-d)$-dimensional cell from $\cY_d$.
\end{enumerate}
\end{proposition}
\begin{proof}
We will show (1) $\Rightarrow$ (2) $\Rightarrow$ (3) $\Rightarrow$ (1). If $w$ and $z$ lie in the same connected component of $\cY_d - \cY_{d+1}$, then there is a zigzag of cells
\[
w \geq v_0 \leq v_1 \geq \cdots \geq v_k \leq z,
\]
lying entirely in $\cY_d - \cY_{d+1}$. By Definition \ref{def:indtrip}, each face relation ($\geq$ or $\leq$) appearing in this zigzag lies in $W_{d}$, and by Lemma \ref{lem:acyclicity}, no such relation can lie in $W_{d-1}$. So by (\ref{eqn:Wd}), the cell $w$ (and indeed, each cell in this zigzag) has its local cohomology $H^\bullet\bL_{d}(v)$ isomorphic to $H^\bullet R[n-d]$. By the argument used while proving Corollary \ref{cor:topstrat}, and recalling that $\dim \cY_d \leq (n-d)$, there exists an $(n-d)$-cell $v \geq  w$ in $\cY_d$ which is isomorphic  in $\bS_d$ to $w$, and hence to every other cell in our zigzag (including $z$). Finally, if there is an $(n-d)$-cell $v$ from $\cY_d$ isomorphic to both $w$ and $z$ in $\bS_d$, then there is a $W_d$-zigzag from $w$ to $z$ obtained by composing zigzags $w \to v \to z$. By Lemma \ref{lem:acyclicity}, all cells in this composite lie in $W_d - W_{d-1}$, so $w$ and $z$ lie in the same connected component of $\cY_d - \cY_{d+1}$.   
\end{proof}

Our next result describes the canonical strata of $\cX$ as isomorphism classes in $\bS_\bullet$.
\begin{proposition}
\label{prop:strattoiso}
For each $d \in [n]$ and canonical $(n-d)$-stratum $\sigma$ of $\cX$, every cell lying in $\sigma$ is isomorphic in $\bS_d$ to some $(n-d)$-cell from $\cY_d$.
\end{proposition}
\begin{proof}
Let $\varnothing = \cL_{-1} \subset \cdots \subset \cL_{d-1}  = \cL$ be the link of $\sigma$ in the sense of Definition \ref{def:strat}. Assume, proceeding via induction, that for each $j < d$ the canonical $(n-j)$-strata of $\cX$ correspond to isomorphism classes in $\bS_j$ of $(n-j)$-cells from $\cY_j$. By Proposition \ref{prop:isocon}, these strata are precisely the connected components of $\cY_j - \cY_{j+1}$. Thus, for each cell $x$ lying in $\sigma$ and $j < d$, we have quasi-isomorphisms
\[
\CBM^\bullet(\Cone\cL_{d-j-1} \times \R^{n-d}) \stackrel{\simeq}{\longrightarrow} \CBM^\bullet(\st(x) \cap \cY_{j}),
\]
where $\st(x)$ the the open star of $x$ in $\cX = \cY_0$. Assuming that $x \geq y$ holds in $\cX$ for another cell $y$ lying in $\sigma$, we would like to prove that $\bL_j^\bullet(x \geq y)$ is a quasi-isomorphism. To this end, consider the following diagram in $\Ch(R)$:
\[
\xymatrixrowsep{0.28in}
\xymatrixcolsep{0.3in}
\xymatrix{
\CBM^\bullet(\Cone\cL_{d-j-1} \times \R^{n-d}) \ar@{->}[r]^-{\simeq} \ar@{->}[rd]_-{\simeq} & \CBM^\bullet(\st(y) \cap \cY_{j}) \ar@{->}[d] \ar@{->}[r]^-{\simeq} & \bL_j^\bullet(y)    \\
																																										& \CBM^\bullet(\st(x) \cap \cY_{j})  \ar@{->}[r]_-{\simeq} & \bL_j^\bullet(x) \ar@{^{(}->}[u]_-{\bL^\bullet(x \geq y)}
}
\]
\noindent The triangle to the left commutes by the refinability constraint of Definition \ref{def:strat} (note that here the downward pointing arrow is the map induced by inclusion), whereas square to the right commutes by the natural contravariant equivalence between singular and cellular compactly-supported cohomology. Being sandwiched in this commuting diagram forces $\bL_j^\bullet(x \geq y)$ to also be a quasi-isomorphism for all $j < d$, so we have $(x \geq y) \in E_{d-1}$. Being an $(n-d)$-stratum, $\sigma$ is an $R$-cohomology manifold, so there are quasi-isomorphisms 
\[
R[n-d]^\bullet \stackrel{\simeq}{\inc} \bL_d^\bullet(x) \stackrel{\simeq}{\inc} \bL_d^\bullet(y). 
\] Since $\sigma$ is a top-dimensional stratum of the subcomplex $\cY_d$, we have $(x' \geq y) \in E_d$ whenever $x' \geq y$ in $\cY_d$ by Proposition \ref{prop:upclose}. Thus, $(x \geq y) \in W_d - W_{d-1}$ by (\ref{eqn:Wd}). Since $\sigma$ is connected, any two cells of $\cX$ that it contains may be connected by a zigzag of cells lying in $\sigma$; and by the preceding argument, each face relation appearing in such a zigzag lies in $W_d - W_{d-1}$. The desired conclusion now follows from Proposition \ref{prop:isocon}.
\end{proof}

The following result establishes the converse of Proposition \ref{prop:strattoiso}.
\begin{proposition}
\label{prop:isotostrat}
For each $d$ in $[n]$, the isomorphism class in $\bS_d$ of any $(n-d)$-dimensional cell from $\cY_d$ is a canonical $(n-d)$-stratum of $\cX$.
\end{proposition}
\begin{proof}
Proceeding once again by induction, we assume that the statement holds for all $j$ in $[d-1]$. Fix an $(n-d)$-dimensional cell $w$ in $\cY_d \subset \cX$, and let $\sigma$ be the collection of all cells which are isomorphic to $w$ in $\bS_d$. We first show that $\sigma$ satisfies the link axiom from Definition \ref{def:strat}. Writing $\st(w)$ and $\lk(w)$ for $w$'s star and link in $\cX$, for each $i \leq d$ we have the intersections
\[
\st_i(w) = \st(w) \cap \cY_{i} \quad \text{ and } \quad \lk_i(w) \cap \cY_{i},
\] which form the star and link of $w$ in $\cY_{i} \subset \cX$. Since $\st_\bullet(w)$ admits a stratified homeomorphism to the corresponding $\Cone\lk_\bullet(w) \times \R^{n-d}$, we get stratified quasi-isomorphisms of singular cochain complexes
\[
\CBM^\bullet(\Cone\lk_i(w) \times \R^{n-d}) \stackrel{\simeq}{\longrightarrow}  \CBM^\bullet(\st_i(w)), 
\]
whose codomains are, in turn, quasi-isomorphic to the local cohomology stalks $\bL_i^\bullet(w)$. If $z$ is any other cell which becomes isomorphic to $w$ in $\bS_d$, then by Proposition \ref{prop:isocon} there a zigzag from $w$ to $z$ whose face relations all lie in $W_d - W_{d-1}$:
\[
w \geq v_0 \leq v_1 \geq \cdots \geq v_{k-1} \leq v_k \geq z,
\]
and hence zigzags of quasi-isomorphisms in $\Ch(R)$ for all $i \in [d]$
\[
\bL_i^\bullet(w) {\to} \bL_i^\bullet(v_0)  {\gets} \cdots \gets \bL_i^\bullet(v_k) {\to} \bL_i^\bullet(z).
\]
Each complex in sight is bounded and free (hence projective), so the backwards-pointing maps admit quasi-inverses (see \cite[Cor 10.4.7]{weibel:94}). Thus, one may turn them around to obtain maps $\bL_i^\bullet(w) \to \bL_i^\bullet(z)$ which fit into a string of quasi-isomorphisms
\[
\CBM^\bullet(\Cone\lk_i(w) \times \R^{n-d}) \to \CBM^\bullet(\st_i(w)) 
\to \bL^\bullet_i(w) \to \bL^\bullet_i(z) 
\to \CBM^\bullet(\st_i(z)).
\]
Setting $\cL_j = \lk_{d-j-1}(w)$ for all $j \in [d-1]$, one obtains a candidate $(d-1)$-dimensional stratified space $\cL$ to serve as the link of $\sigma$. 

Turning now to the frontier axiom, for each $i < d$ let $\bK_i^\bullet$ be the quotient complex $\bL^\bullet_i/\bL^\bullet_{i+1}$, suitably restricted so that it assigns zero stalks to all cells lying outside $\sigma$. Over each cell $z$ in $\sigma$, the cohomology $H^\bullet\bK_i(z)$ coincides with the compactly-supported cohomology of the difference $\st_i(z) - \st_{i+1}(z)$. By construction, the closure of a canonical $(n-i)$-stratum $\tau$ in $\cX$ intersects $\sigma$ at the cell $z$ if and only if the cohomology $H^{n-i}\bK_i(z)$ is non-trivial. By definition, $\bK_i^\bullet$ fits into a short exact sequence (of complexes of cosheaves over $\sigma$):
\[
0 \to \bL^\bullet_{i+1} \inc \bL^\bullet_i \pro \bK_i^\bullet \to 0.
\]
Applying the five-lemma \cite[Ex 1.3.3]{weibel:94} to the resulting long exact sequence, we note that all the extension maps of $\bK_i^\bullet$ induce isomorphisms on cohomology. Thus, if $H^{n-i}\bK_i$ is non-zero at $z$, then it is also non-zero at every other cell in $\sigma$, whence every cell in $\sigma$ admits a co-face from $\tau$. Thus, $\sigma \subset \overline{\tau}$, as desired.

Finally, to see that $\sigma$ is canonical, we use Proposition \ref{prop:strattoiso} to pick an $(n-d)$-cell $w$ from $\cY_d$ so that all cells in $\sigma$ are isomorphic in $\bS_d$ to $w$. Proposition \ref{prop:isocon} now guarantees that no additional cells may be added to $\sigma$ while preserving its connectedness.
\end{proof}

The three preceding propositions establish that the sequence $\cY_\bullet \subset \cX$ from Definition \ref{def:indtrip} constitutes the canonical stratification of $\cX$. The localization functors $\Face\op(\cX) \to \bS_\bullet$ and the universal functors $\bS_\bullet \to \bS_{\bullet+1}$ arising from the inclusions $W_\bullet \subset W_{\bullet+1}$ fit into a commutative diagram
\[
\xymatrixrowsep{0.25in}
\xymatrixcolsep{0.35in}
\xymatrix{
\Face\op(\cX) \ar@{->}[d] \ar@{->}[rd] \ar@{->}[rrrd] & & &  \\
\bS_0 \ar@{->}[r] & \bS_1 \ar@{->}[r] & \cdots \ar@{->}[r] & \bS_n = \bS,
}
\]
which one may view as analogous to an injective resolution of the functor $\Face\op(\cX) \to \bS_0$ from Section \ref{sec:top}. Proceeding from left to right in the sequence above, we obtain canonical strata of decreasing dimension by examining isomorphism classes in $\bS_\bullet$. We collect these results into the following expanded version of Theorem \ref{thm:main}. 
\begin{theorem}
\label{thm:full}
Given a finite regular $n$-dimensional CW complex $\cX$, let $\cY_\bullet$ and $W_\bullet$ be as in Definition \ref{def:indtrip}, and let $\bS_\bullet$ denote the localization of the face poset $\Face\op(\cX)$ about $W_\bullet$. For each $d \in [n]$, the canonical $(n-d)$-strata of $\cX$ correspond bijectively with isomorphism classes of $(n-d)$-dimensional cells from $\cY_d$ in $\bS_k$ for any $k \geq d$.
\end{theorem}

The morphisms in the last category $\bS = \bS_n$ also recover the frontier partial order among canonical strata.
\begin{proposition}
\label{prop:frontier}
The frontier relation $\tau \succeq \sigma$ holds among two canonical strata of $\cX$ if and only if there is a morphism in $\bS$ from a cell lying in $\tau$ to a cell lying in $\sigma$.
\end{proposition} 
\begin{proof}
If $\tau \succeq \sigma$ holds, then there are cells $w$ and $z$ in $\tau$ and $\sigma$ respectively so that $w \geq z$ in $\Face\op(\cX)$. This face relation includes into $\bS$ as a morphism from $w$ to $z$. Conversely, given cells $w$ and $z$ in $\tau$ and $\sigma$ respectively, assume the existence of a $W$-zigzag
\[
\zeta = (w \geq y_0 \leq x_0 \geq \cdots \geq y_k \leq x_k \geq z)
\]
and let $\varphi_\zeta:[k] \to [n]$ be the associated monotone function from Lemma \ref{lem:acyclicity}. By construction, this zigzag contains cells in canonical strata $\tau_i$ of dimension $n-\varphi_\zeta(i)$ for all $i$ in $[k]$. If $\varphi_\zeta(i) = \varphi_\zeta(i+1)$ then $\tau_i = \tau_{i+1}$, otherwise $\tau_i \succ \tau_{i+1}$ (because $x_i$ is a co-face of $y_{i+1}$). Thus, there is a descending sequence of canonical strata 
\[
\tau \succeq \tau_0 \succeq \tau_1 \succeq \cdots \succeq \tau_k \succeq \sigma,
\]
so we obtain $\tau \succeq \sigma$ as desired.
\end{proof}

The argument above shows that the zigzag paths which form morphisms in $\bS$ may remain indefinitely in a single stratum or descend to lower-dimensional adjacent strata, but they can never ascend to higher strata. In other words, $\bS$ is a (cellular, $1$-categorical) version of the {\em entrance path category} associated to the canonical stratification of $\cX$ --- see \cite[Def 3.1]{nanda:15}, \cite[Sec 7]{treumann:09} or \cite[Sec 2]{woolf:09}.  

\section{Algorithms} \label{sec:comp}

From a computational perspective, the local cohomology $\bL^\bullet$ of a finite regular CW complex enjoys the obvious, but enormous, advantage of being {\em local}. One only ever needs to construct cochain complexes corresponding to open stars of cells (rather than holding the entire complex in system memory). Since each such star-complex may be processed  independently of the others, the construction of $\bL^\bullet$ and the computation of $H^\bullet\bL$ are inherently distributable operations. 

\subsection{Subroutines}

There are several avenues for efficiently storing a CW complex for the purpose of computing its cohomology: one may employ a {\em Hasse graph}, whose vertices are cells and whose edges, usually adorned with an orientation, connect neighboring cells across codimension one. Alternately, one may directly store (sparse avatars of) {\em coboundary matrices} over $R$, whose nonzero entries encode adjacencies between cells of codimension one. In the presence of extra structure -- for instance, if the complex involved is simplicial or cubical -- the additional rigidity affords access to highly optimized data structures such as simplex trees or bitmaps. Rather than tethering the prospective implementer to a fixed and possibly inconvenient representation, we will present all our algorithms at a high level and work directly with the poset of cells. 

\begin{table}[h!]
\centering
{\bf Algorithm: }{\tt UpSet}$(P,u,v)$ \label{alg:getupset} \\
{\bf In: }A finite poset $P$ and elements $u, v \in P$ \\
{\bf Out: } The subposet $P_v^u$ of elements $\geq v$ but not $\geq u$ \\
{
\begin{tabular}{c|l}
\hline
01 & {\bf initialize} an empty {\em queue} $Q$ of elements \\
02 & {\bf add} $v$ to $Q$ \\
03 & {\bf while} $Q$ is nonempty \\
04 & \tab {\bf remove} $w$ from $Q$ \\
05 & \tab {\bf if} $w$ is not in $P_v^u$ \\
06 & \tab \tab {\bf add} $w$ to $P_v^u$ \\
07 & \tab \tab {\bf for each} element $x > w$ in $P$ with $x \not\geq u$ \\
08 & \tab \tab \tab {\bf add} $x$ to $Q$ (uniquely) \\
09 & {\bf return} $P_v^u$ \\
\hline
\end{tabular}
}
\end{table}

Our first subroutine {\tt UpSet} employs a minor modification of breadth-first search on a finite poset $P$ to extract the subposet of all elements which lie above a fixed element $v$ but {\em not} above a fixed element $u$. In all instances of interest here, our input $P$ will be the face poset $\Face\op(\cX)$ of a finite regular CW complex $\cX$, while $u$ and $v$ will be adjacent cells across codimension one: 
\[
u >_1 v \text{ if } u > v \text{ and } \dim u - \dim v = 1.
\] By definition, the output $P_v^u$ in such cases is precisely the subposet of $\Face\op(\cX)$ containing cells which lie in the difference of open stars $\st(v) - \st(u)$. We also reserve the right to set $u = \varnothing$, so that line 07 does not check $x$ against $u$ and the output is just the poset of cells in $\st(v)$.

The next required subroutine is {\tt Cohom}, which accepts a poset of cells as input and returns the sequence of $R$-modules corresponding to the cohomology of the associated cochain complex. When $R$ is a field, the cohomology modules are all vector spaces and in this case the output can just be a sequence of integers storing their dimensions. If $R$ is the ring of integers, then one has to encode the torsion subgroups (if any) in addition to the ranks of the free parts in every dimension. In either case, since the efficient computation of (co)homology has been extensively discussed elsewhere \cite{curry:ghrist:nanda:16}, \cite{hmmn:14}, \cite{henselman:ghrist:16}, \cite{mischaikow:nanda:13}, we will treat this subroutine as a black box and not explicitly write it down here. 

\subsection{Description}

Our main algorithm {\tt StratCast} accepts as input the poset of cells in a finite $n$-dimensional regular CW complex $\cX$ and uses the constructions of Definition \ref{def:indtrip} to assign each cell $x$ a number $\codim(x)$ in $[n] = \{0,1,\cdots,n\}$ so that $x$ lies in a canonical stratum of dimension $n-\codim(x)$. Once codimensions have been determined for all cells, finding the actual canonical strata reduces to the simple task of locating connected components of a fixed codimension.  

\begin{table}[h!]
\centering
{\bf Algorithm: }{\tt StratCast}$(\Face\op(\cX))$ \label{alg:stratcast} \\
{\bf In: }The face poset of a finite $n$-dim regular CW complex \\
{\bf Out: } Assigns canonical codimensions to cells \\
{
\begin{tabular}{c|l}
\hline
01 & {\bf for each} $d$ in $(0,1, \ldots, n-1, n)$ \\
02 & \tab {\bf for each} cell $w$ in $\Face\op(\cX)$ \hfill $\fiber$ {\em dist} \\
03 & \tab \tab {\bf assign} $h^\bullet(w) = $ {\tt Cohom}({\tt UpSet}($\Face\op(\cX),\varnothing,w$)) \\
04 & \tab {\bf for each} $(x >_1 y)$ in $\Face\op(\cX)$ with $\ell(x >_1 y) = d-1$ \hfill $\fiber$ {\em dist} \\
05 & \tab \tab {\bf assign} $c^\bullet(x,y) = $ {\tt Cohom}({\tt UpSet}($\Face\op(\cX),y,x$)) \\
06 & \tab \tab {\bf if} $c^\bullet(x,y)$ is trivial \\
07 & \tab \tab \tab {\bf assign} $\ell(x >_1 y) = d$ \\
08 & \tab {\bf assign} $\codim(z) = d$ to each $(n-d)$-cell $z$ in $\Face\op(\cX)$ \\
09 & \tab {\bf for each} $i$-cell $u$ in $\Face\op(\cX)$ with $i$ in $(n-d-1,\ldots, 1, 0)$  \hfill $\fiber$ {\em dist} \\
10 & \tab \tab {\bf if} $h^{n-d}(u) \simeq R$ and all other $h^\bullet(u) = 0$ \\
11 & \tab \tab \tab {\bf if} $\codim(v) = d$ and $\ell(v >_1 u) = d$ for all $v >_1 u$ in $\Face\op(\cX)$ \\
12 & \tab \tab \tab \tab {\bf assign} $\codim(u) = d$ \\
13 & \tab {\bf remove} all cells from $\Face\op(\cX)$ with $\codim = d$ \\
\hline
\end{tabular}
}
\end{table}

The main loop (in line 01) increments the current codimension $d$ from $0$ up to $n$; as it executes each iteration, the cells lying on $(n-d)$-dimensional canonical strata are identified and removed from $\cX$, so that in the $d$-th iteration we pass from $\cY_d$ to $\cY_{d+1}$ in the language of Definition \ref{def:indtrip}. There are three secondary loops (at lines 02, 04 and 09); all three are trivially distributable across several processors, as indicated by the {\em dist} comments in those lines. The first loop in line 04 is easiest to explain: it uses our subroutines {\tt UpSet} and {\tt Cohom} to compute the compactly-supported cohomology of each open star $\st(w) \cap \cY_d$ where $w$ is a cell in $\cY_d \subset \cX$. 

The second loop from line 04 computes whether or not inclusions of open stars $\st(x) \subset \st(y)$ for $x >_1 y$ induce isomorphisms on compactly-supported cohomology by examining their cokernels. In other words, we use the fact that the map induced by $\bL_d^\bullet(x >_1 y)$ (namely, the inclusion $\st(x) \cap \cY_d \inc \st(y) \cap \cY_d$) on compactly-supported cohomology is an isomorphism if and only if the difference $(\st(y) - \st(x)) \cap \cY_d$ has trivial compactly-supported cohomology. We begin by assigning to each face relation $(x >_1 y)$ the number $\ell(x >_1 y) = -1$. As the algorithm executes, $\ell(x >_1 y)$ gets incremented to the largest integer $k$ so that $(x >_1 y)$ lies in $E_k$, where $E_\bullet$ denotes the sets from (\ref{eqn:Ed}). Since this loop is completely independent of the loop in line 02, it may be executed simultaneously with that loop. 

The third loop, which spans lines 09 through 12, must be executed after the first two intermediate loops have terminated. It uses the cohomology groups $h^\bullet$ computed by the loop in line 02 and the $\ell$-values computed by the loop in line 04 to select cells lying on canonical strata of dimension $(n-d)$ via (\ref{eqn:Wd}). Note that line 10 ensures the correct dimensionality and line 11 enforces upward closure. Line 13 removes all cells lying on canonical $(n-d)$-strata from $\cY_d$, so we are left with $\cY_{d+1}$ and the outer loop of line 01 enters its next iteration.

\subsection{Complexity}

The intermediate loops in line 02 and line 04 of {\tt StratCast} require breadth-first search (for constructing open stars via {\tt UpSet}) and matrix diagonalization through row and column operations (for computing cohomology with {\tt Cohom}). Let $m$ be the number of cells in $\cX$, and define the additional {\em star size} complexity parameter
\begin{align}
\label{eqn:starsize}
p = \max \set{ \text{number of cells in } \st(x) \mid x \text{ is a cell in } \cX}
\end{align}
(it suffices to take the maximum over zero-dimensional cells, since those have the largest stars). The worst-case complexity of computing cohomology of any open star, or a difference of open stars, via row and column operations is thus $O(p^3)$ provided we incur a constant cost performing ring operations in $R$. The cost of running {\tt Cohom} clearly dominates the $O(p^2)$ cost of executing breadth-first search in {\tt UpSet}. Since the loops of line 02 and 04 may be executed in a completely distributed fashion, their combined cost within a single iteration of the outer loop of line 01 is $O(p^3)$.

The third intermediate loop (on line 09) may be reinterpreted as two nested loops: an outer $i$-loop which decrements dimension from $(n-d-1)$ to $0$, and an inner $u$-loop which examines all remaining cells of dimension $i$. Since we require knowledge of {\em codim} values of all cells $v >_1 u$ while processing each $u$, only the $u$-loop is actually distributable. During each iteration of the $u$-loop, one incurs an $O(p)$ cost while checking all remaining cells $v >_1 u$ in line 11, and an $R$-dependent cost of testing $R$-module isomorphisms in line 10. At least when $R$ is a finite field, the rationals or the integers, checking whether a given module has rank one or zero incurs $O(1)$ cost. Thus, in typical cases we expect line 11 to dominate the burden of executing the $u$-loop. Noting that the $i$-loop runs at most $n$ times regardless of $d$, each iteration of the loop on line 09 incurs a computational cost of $O(np)$. 

Since all cells with {\em codim} $ = d$ are identified on lines 08 and 12 in a given iteration of the outer $d$-loop of {\tt StratCast}, we assume that removing them from $\cX$ incurs no additional cost. Putting all our estimates together, each iteration of the $d$-loop incurs a worst-case cost of $O(p^3 + np)$. Noting that this loop executes exactly $(n+1)$ times, the total complexity of running {\tt StratCast} -- assuming that the number of available processors exceeds the number of cells $m$ plus the maximum number $mp$ of codimension-one face relations in $\cX$ -- is $O\left((n+1)(p^3+np)\right)$. And once {\tt StratCast} has terminated, one may find all desired canonical strata of $\cX$ by computing connected components of cells with the same {\em codim} value. This is linear in the number of cells in $\cX$, so we obtain the following result.

\begin{proposition}
\label{prop:complexity}
Let $\cX$ be an $n$-dimensional regular CW complex with $m$ cells and star size $p$ as in (\ref{eqn:starsize}). The time complexity of computing its canonical stratification with coefficients in a (commutative, unital) ring $R$ via {\tt StratCast} is 
\[
O\left((n+1)(p^3+np) + m\right),
\]
provided that 
\begin{enumerate}
\item ring operations in $R$ incur $O(1)$ cost,
\item isomorphism-testing against $0$ and $R$ in $\Mod(R)$ is $O(1)$, and
\item the number of available processors exceeds $m(p+1)$.
\end{enumerate}
\end{proposition}
The first two conditions above are satisfied by typical choices of $R$ (such as the integers or finite fields). In practice, one expects to have $m \gg (n+1)(p^3+np)$, so the observed cost of computing canonical strata is essentially linear in the size of $\cX$.

\bibliographystyle{abbrv}
\bibliography{stratbib}

\end{document}